\documentclass[11pt]{amsart}

\usepackage{amsmath,amssymb,amsthm}

\newtheorem{theorem}{Theorem}
\newtheorem{lemma}{Lemma}
\newtheorem{corollary}[theorem]{Corollary}
\theoremstyle{remark}

\newcommand{\Z}{\mathbb Z}
\newcommand{\qbinom}[3][]{\genfrac{[}{]}{0pt}{}{#2}{#3}_{#1}}
\newcommand{\sieve}[2]{\mathcal S_{#1,#2}}

\title{Solutions for Hecke Sum Questions of Banerjee and Bringmann}

\author{George E. Andrews}
\address{The Pennsylvania State University, University Park, Pennsylvania 16802}
\email{andrews@math.psu.edu}

\author{Mohamed El Bachraoui}
\address{Dept. Math. Sci,
United Arab Emirates University, PO Box 15551, Al-Ain, UAE}
\email{melbachraoui@uaeu.ac.ae}

\thanks{First author partially supported by Simons Foundation Grant 633284.}
\date{May 2026}

\begin{document}

\begin{abstract}
We previously conjectured a Hecke-type identity for a series arising from two-color partitions. Banerjee and Bringmann proved the conjecture using modular methods and asked for a direct proof using Bailey pairs. They also suggested investigating the odd part of the series. We answer both questions by applying Bailey's transform to a Bailey pair of Andrews and a conjugate Bailey pair of Lovejoy. Our main theorem gives formulas for the even and odd parts of a series $P_a(q)$ depending on a parameter $a$. The specialization $a=1$ gives the conjectured formula and a companion formula for the odd part.
\end{abstract}


\maketitle

\section{Introduction}

We use the standard $q$-series notation as for instance in~\cite[Chapter 1]{GasperRahman}
\[
(a;q)_0:=1,\qquad (a;q)_n:=\prod_{j=0}^{n-1}(1-aq^j),\qquad
(a;q)_\infty:=\prod_{j=0}^{\infty}(1-aq^j),
\]
and
\[
(a_1,\ldots,a_r;q)_n:=(a_1;q)_n\cdots(a_r;q)_n.
\]
For a series $F(q)=\sum_{n\ge0} A(n)q^n$, put
\[
F(q)\big|\sieve{N}{r}:=\sum_{n\equiv r\,({\rm mod}\,N)} A(n)q^n.
\]

In \cite{AndrewsBachraoui}, we studied the series
\[
S(q):=\sum_{m\ge0}\frac{(q;q^2)_{m}^2q^{2m}}{(-q^2;q^2)_{m}}
\]
in connection with two-color partitions with odd smallest part. We conjectured a Hecke-type formula for the even part of $(q^4;q^4)_\infty S(q)$. Banerjee and Bringmann \cite{BanerjeeBringmann} proved this conjecture using indefinite theta functions and modular forms. They asked whether a direct proof using Bailey pairs could be found and also suggested studying the odd part.

Our proof uses the Bailey pair in \cite[Lemma~3]{AndrewsBaileyChains}
and the conjugate Bailey pair in \cite[Eq.~(1.10)]{LovejoyPartialTheta}.
We recall their specializations below and include an elementary verification
of the conjugate pair needed here.

We prove the following identities with an additional parameter $a$. Let $a$ be an indeterminate, or a nonzero complex number, and define
\begin{equation}\label{Sa-def}
S_a(q):=\sum_{m\ge0}\frac{(aq;q^2)_m(q/a;q^2)_m q^{2m}}{(-q^2;q^2)_m},
\qquad
P_a(q):=(q^4;q^4)_\infty S_a(q).
\end{equation}
Thus $S_1(q)=S(q)$.

The following theorem gives the two parity parts of $P_a(q)$.

\begin{theorem}\label{thm:twovar}
For $|q|<1$,
\begin{equation}\label{twovar-even}
P_a(q)\big|\sieve{2}{0}
=
\sum_{n\ge0}(-1)^nq^{6n^2+4n}(1+q^{4n+2})
\sum_{j=-n}^{n}(-1)^j a^{2j}q^{-2j^2},
\end{equation}
and
\begin{equation}\label{twovar-odd}
P_a(q)\big|\sieve{2}{1}
=
\sum_{n\ge0}(-1)^{n+1}q^{6n^2+10n+3}(1+q^{4n+4})
\sum_{j=0}^{n}(-1)^j
\big(a^{2j+1}+a^{-2j-1}\big)q^{-2j(j+1)}.
\end{equation}
\end{theorem}

Putting $a=1$ in Theorem~\ref{thm:twovar} gives the two formulas sought for $S(q)$. The first is the Hecke-type identity conjectured in \cite{AndrewsBachraoui} and proved in \cite{BanerjeeBringmann}.

\begin{corollary}\label{cor:even}
For $|q|<1$,
\begin{equation}\label{even-main}
(q^4;q^4)_\infty S(q)\big|\sieve{2}{0}
=
\sum_{n\ge0}(-1)^nq^{6n^2+4n}(1+q^{4n+2})
\sum_{|j|\le n}(-1)^jq^{-2j^2}.
\end{equation}
\end{corollary}

The second formula gives the odd part of $(q^4;q^4)_\infty S(q)$.

\begin{corollary}\label{cor:odd}
For $|q|<1$,
\begin{equation}\label{odd-main}
(q^4;q^4)_\infty S(q)\big|\sieve{2}{1}
=
2\sum_{n\ge0}(-1)^{n+1}q^{6n^2+10n+3}(1+q^{4n+4})
\sum_{j=0}^{n}(-1)^jq^{-2j(j+1)}.
\end{equation}
\end{corollary}

\section{Two lemmas}

Recall that $(\alpha_j,\beta_n)_{j,n\ge0}$ is a Bailey pair relative
to $1$, with base $q^2$, if
\begin{equation}\label{bailey-pair-def}
\beta_n=\sum_{j=0}^n
\frac{\alpha_j}{(q^2;q^2)_{n-j}(q^2;q^2)_{n+j}}.
\end{equation}
A conjugate Bailey pair $(\gamma_j,\delta_n)_{j,n\ge0}$, with the
same base, satisfies
\begin{equation}\label{conjugate-bailey-def}
\gamma_j=\sum_{n=j}^\infty
\frac{\delta_n}{(q^2;q^2)_{n-j}(q^2;q^2)_{n+j}}.
\end{equation}
Bailey's transform gives
\begin{equation}\label{bailey-transform}
\sum_{n\ge0}\beta_n\delta_n=\sum_{j\ge0}\alpha_j\gamma_j,
\end{equation}
whenever the sums converge; see \cite[Chapter~3]{AndrewsQSeries} and
\cite{WarnaarBailey}.

The $q$-binomial coefficient is defined by~\cite[Definition 1]{AndrewsTheory}
\[
\qbinom[q]{M}{K}:=
\begin{cases}
\displaystyle \frac{(q;q)_M}{(q;q)_K(q;q)_{M-K}},&0\le K\le M,\\[6pt]
0,&\text{otherwise}.
\end{cases}
\]

We shall need the finite $q$-binomial theorem~\cite[Eq. (3.3.6)]{AndrewsTheory}
\begin{equation}\label{finite-qbinom}
(z;q)_m=\sum_{j=0}^m(-1)^jq^{j(j-1)/2}\qbinom[q]{m}{j}z^j
\end{equation}
and Euler's expansion~\cite[Eq. (2.2.6)]{AndrewsTheory}
\begin{equation}\label{euler-expansion}
(z;q)_\infty=\sum_{m\ge0}\frac{(-1)^mq^{m(m-1)/2}z^m}{(q;q)_m}.
\end{equation}
We also use the $q$-binomial theorem in the form
\begin{equation}\label{reciprocal-qbinom}
\sum_{\ell\ge0}\frac{x^\ell}{(q;q)_\ell}=\frac1{(x;q)_\infty}
\qquad (|x|<1).
\end{equation}
In the proof below, Euler's expansion will be used in the special form
\begin{equation}\label{euler-used}
(-zq^{k+1};q)_\infty
=\sum_{m\ge0}\frac{z^mq^{m(k+1)+m(m-1)/2}}{(q;q)_m}.
\end{equation}

We first record a finite identity needed in the next proof.

\begin{lemma}\label{lem:pentagonal}
For $N\ge0$, let
\begin{equation}\label{DN}
D_N(q):=
\sum_{j=0}^{\lfloor N/2\rfloor}
(-1)^jq^{j(3j+1-2N)/2}\qbinom[q]{N-j}{j}.
\end{equation}
Then
\begin{equation}\label{finite-pentagonal}
D_N(q)=
\begin{cases}
(-1)^s q^{-s(3s-1)/2},&N=3s,\\[3pt]
(-1)^s q^{-s(3s+1)/2},&N=3s+1,\\[3pt]
0,&N=3s+2.
\end{cases}
\end{equation}
\end{lemma}

\begin{proof}
The initial values are
\[
D_0(q)=1,\qquad D_1(q)=1,\qquad D_2(q)=0.
\]
We shall prove that, for $N\ge3$,
\begin{equation}\label{DN-recurrence}
D_N(q)=-q^{2-N}D_{N-3}(q).
\end{equation}
The formula \eqref{finite-pentagonal} follows at once by iterating this recurrence.

We use the elementary $q$-Pascal identities~\cite[Eqs. (3.3.3)-(3.3.4)]{AndrewsTheory}
\[
\qbinom[q]{M}{K}=\qbinom[q]{M-1}{K}+q^{M-K}\qbinom[q]{M-1}{K-1}
\]
and
\[
(1-q^K)\qbinom[q]{M}{K}=(1-q^{M-K+1})\qbinom[q]{M}{K-1}.
\]
Applying the first identity to $\qbinom[q]{N-j}{j}$ gives two sums. The term
with $j=0$ in the second sum is zero. Replacing $j$ by $j+1$ in that sum gives
\[
-\sum_j(-1)^jq^{j(3j+3-2N)/2}\qbinom[q]{N-2-j}{j}.
\]
Hence
\[
D_N(q)=
\sum_j(-1)^jq^{j(3j+1-2N)/2}\qbinom[q]{N-1-j}{j}
-
\sum_j(-1)^jq^{j(3j+3-2N)/2}\qbinom[q]{N-2-j}{j}.
\]
Apply the first $q$-Pascal identity once more to
$\qbinom[q]{N-1-j}{j}$ in the first sum. We obtain
\begin{align*}
D_N(q)
&=\sum_{j\ge0}(-1)^jq^{j(3j+1-2N)/2}(1-q^j)
\qbinom[q]{N-2-j}{j}\\
&\quad+\sum_{j\ge1}(-1)^jq^{j(3j+1-2N)/2+N-1-2j}
\qbinom[q]{N-2-j}{j-1}.
\end{align*}
By the second $q$-Pascal identity,
\[
(1-q^j)\qbinom[q]{N-2-j}{j}
=(1-q^{N-2j-1})\qbinom[q]{N-2-j}{j-1}.
\]
The term $-q^{N-2j-1}$ cancels the corresponding term in the second
sum. Consequently,
\[
D_N(q)=\sum_{j\ge1}(-1)^jq^{j(3j+1-2N)/2}
\qbinom[q]{N-2-j}{j-1}.
\]
Replacing $j$ by $j+1$ now gives
\[
D_N(q)=
-q^{2-N}
\sum_{j=0}^{\lfloor (N-3)/2\rfloor}
(-1)^jq^{j(3j+1-2(N-3))/2}\qbinom[q]{N-3-j}{j}.
\]
The sum on the right is $D_{N-3}(q)$.  This proves \eqref{DN-recurrence}, and hence the lemma.
\end{proof}

The next identity will be used to verify the specialization of
Lovejoy's conjugate Bailey pair needed in the proof of the theorem.
We include an elementary proof for completeness.

\begin{lemma}\label{lem:tail}
For $|q|<1$,
\begin{equation}\label{tail}
\begin{aligned}
&(q;q)_\infty
\sum_{k\ge0}
\frac{q^k(z^2q^{k+1};q)_k(-zq^{k+1};q)_\infty}{(q;q)_k} \\
&\hspace{35mm}=
\sum_{n\ge0}(-1)^nz^{3n}q^{3n^2+2n}(1+zq^{2n+1}).
\end{aligned}
\end{equation}
\end{lemma}

\begin{proof}
Let the left side of \eqref{tail} be $L(z)$.  By \eqref{finite-qbinom}, with the variable there replaced by $z^2q^{k+1}$, and by Euler's expansion in the form \eqref{euler-used},
\[
(z^2q^{k+1};q)_k
=
\sum_{j=0}^k(-1)^jz^{2j}q^{j(k+1)+j(j-1)/2}\qbinom[q]{k}{j},
\]
and
\[
(-zq^{k+1};q)_\infty
=
\sum_{m\ge0}\frac{z^mq^{m(k+1)+m(m-1)/2}}{(q;q)_m}.
\]
Substituting these expansions into $L(z)$ and using
\[
\frac{\qbinom[q]{k}{j}}{(q;q)_k}
=\frac1{(q;q)_j(q;q)_{k-j}},
\]
we obtain
\begin{align*}
L(z)
&=(q;q)_\infty\sum_{j,m\ge0}
\frac{(-1)^jz^{2j+m}
q^{(3j^2+3j+m^2+2jm+m)/2}}{(q;q)_j(q;q)_m}
\sum_{\ell\ge0}\frac{q^{\ell(j+m+1)}}{(q;q)_\ell},
\end{align*}
where we have put $k=j+\ell$. By \eqref{reciprocal-qbinom},
\[
(q;q)_\infty\sum_{\ell\ge0}
\frac{q^{\ell(j+m+1)}}{(q;q)_\ell}
=\frac{(q;q)_\infty}{(q^{j+m+1};q)_\infty}
=(q;q)_{j+m}.
\]
Thus
\[
L(z)=
\sum_{j,m\ge0}(-1)^jz^{2j+m}
q^{(3j^2+3j+m^2+2jm+m)/2}
\qbinom[q]{j+m}{j}.
\]
To find the coefficient of $z^N$, put $m=N-2j$. Then
\[
\frac{3j^2+3j+m^2+2jm+m}{2}
=\frac{N(N+1)}2+\frac{j(3j+1-2N)}2,
\]
and $\qbinom[q]{j+m}{j}=\qbinom[q]{N-j}{j}$. Therefore
\begin{equation}\label{coeff-L}
[z^N]L(z)=q^{N(N+1)/2}D_N(q).
\end{equation}
By Lemma \ref{lem:pentagonal}, the only nonzero coefficients of $L(z)$ are
\[
[z^{3n}]L(z)=(-1)^nq^{3n^2+2n},
\qquad
[z^{3n+1}]L(z)=(-1)^nq^{3n^2+4n+1}.
\]
This is exactly \eqref{tail}.
\end{proof}

\section{Proof of Theorem~\ref{thm:twovar}}

\begin{proof}

We use the Bailey pair in \cite[Lemma~3]{AndrewsBaileyChains},
with $q$ there replaced by $q^2$ and the parameter $z$ replaced by $aq$.
For completeness, we verify this specialization using the following
case of MacMahon's finite form of Jacobi's triple product
\cite[p.~75]{MacMahon}:
\begin{equation}\label{macmahon-finite-jacobi}
(zq;q)_M(z^{-1};q)_N
=
\sum_{j=-N}^{M}(-1)^j q^{j(j+1)/2}z^j
\qbinom[q]{M+N}{N+j}.
\end{equation}
Taking $M=N$ and $z=a/q$ in \eqref{macmahon-finite-jacobi} gives
\begin{equation}\label{finite-jacobi}
(a;q)_N(q/a;q)_N
=
\sum_{j=-N}^{N}(-1)^j a^j q^{j(j-1)/2}
\qbinom[q]{2N}{N+j}.
\end{equation}
Now apply \eqref{finite-jacobi} with $q$ replaced by $q^2$ and with $a$ there replaced by $aq$.  This gives
\begin{equation}\label{twovar-finite-jacobi}
(aq;q^2)_N(q/a;q^2)_N
=
\sum_{j=-N}^{N}(-1)^j a^j q^{j^2}
\qbinom[q^2]{2N}{N+j}.
\end{equation}
Divide \eqref{twovar-finite-jacobi} by $(q^2;q^2)_{2N}$. By the definition of the $q$-binomial coefficient, the sequence
\begin{equation}\label{beta-a}
\beta_n(a):=\frac{(aq;q^2)_n(q/a;q^2)_n}{(q^2;q^2)_{2n}}
\end{equation}
together with the sequence
\begin{equation}\label{alpha-a}
\alpha_0(a)=1,
\qquad
\alpha_j(a)=(-1)^jq^{j^2}\big(a^j+a^{-j}\big)\quad(j\ge1).
\end{equation}
is a Bailey pair relative to $1$, with base $q^2$.

We next use the conjugate Bailey pair of
Lovejoy \cite[Eq.~(1.10)]{LovejoyPartialTheta}, with $q$ there replaced
by $q^2$ and his parameter set equal to $-1$. Put
\begin{equation}\label{delta-def}
\delta_n=(q^4;q^4)_\infty q^{2n}
\frac{(q^2;q^2)_{2n}}{(-q^2;q^2)_n}.
\end{equation}
For $j\ge0$ set
\begin{align}
\gamma_j
&=\sum_{n=j}^{\infty}
\frac{\delta_n}{(q^2;q^2)_{n-j}(q^2;q^2)_{n+j}} \notag\\
&=q^{2j}(q^2;q^2)_\infty
\sum_{k\ge0}
\frac{q^{2k}(q^{4j+2k+2};q^2)_k(-q^{2j+2k+2};q^2)_\infty}
{(q^2;q^2)_k}. \label{gamma-tail-sum}
\end{align}
To obtain \eqref{gamma-tail-sum}, put $n=j+k$ and observe that
\[
\frac{(q^2;q^2)_{2j+2k}}{(q^2;q^2)_k(q^2;q^2)_{2j+k}}
=\frac{(q^{4j+2k+2};q^2)_k}{(q^2;q^2)_k}.
\]
Moreover,
\[
(q^4;q^4)_\infty=(q^2;q^2)_\infty(-q^2;q^2)_\infty.
\]
It follows that
\[
\frac{(q^4;q^4)_\infty}{(-q^2;q^2)_{j+k}}
=(q^2;q^2)_\infty(-q^{2j+2k+2};q^2)_\infty,
\]
which proves \eqref{gamma-tail-sum}.
Lemma \ref{lem:tail}, with $q$ replaced by $q^2$ and $z=q^{2j}$, gives
\begin{equation}\label{gamma-eval}
\gamma_j=q^{2j}
\sum_{m\ge0}(-1)^mq^{6m^2+(6j+4)m}(1+q^{2j+4m+2}).
\end{equation}
Thus \eqref{delta-def} and \eqref{gamma-eval} give the stated specialization
of Lovejoy's conjugate Bailey pair. Lemma~\ref{lem:tail} supplies a direct
verification of its evaluation.

Since
\[
P_a(q)=\sum_{n\ge0}\beta_n(a)\delta_n,
\]
Bailey's transform \eqref{bailey-transform}, with $q$ replaced by $q^2$, gives
\begin{align}
P_a(q)
&=\sum_{j\ge0}\alpha_j(a)\gamma_j \notag\\
&=\sum_{j\in\Z}(-1)^j a^j q^{j^2}\gamma_{|j|}. \label{master-sum}
\end{align}
This is the basic identity from which both parity parts follow.

Since \eqref{gamma-eval} contains only even powers of $q$, the term
$q^{j^2}\gamma_{|j|}$ has the same parity as $j$. Therefore the even
part of \eqref{master-sum} is obtained by putting $j=2k$:
\begin{equation}\label{twovar-even-first}
P_a(q)\big|\sieve{2}{0}
=
\sum_{k\in\Z}a^{2k}q^{4k^2}\gamma_{2|k|}.
\end{equation}
For $k\ge0$, \eqref{gamma-eval} gives
\[
q^{4k^2}\gamma_{2k}
=
q^{4k^2+4k}
\sum_{m\ge0}(-1)^mq^{12km+6m^2+4m}(1+q^{4k+4m+2}).
\]
Putting $n=k+m$, we have
\[
4k^2+4k+12km+6m^2+4m=6n^2+4n-2k^2,
\qquad 4k+4m+2=4n+2.
\]
Since $(-1)^m=(-1)^{n+k}$, we obtain
\begin{equation}\label{Er-tail}
q^{4k^2}\gamma_{2k}
=
(-1)^kq^{-2k^2}
\sum_{n\ge k}(-1)^nq^{6n^2+4n}(1+q^{4n+2})
\qquad(k\ge0).
\end{equation}
The same formula with $k$ replaced by $|k|$ holds for all integers $k$.
Substituting in \eqref{twovar-even-first} and interchanging the sums gives
\[
P_a(q)\big|\sieve{2}{0}
=
\sum_{n\ge0}(-1)^nq^{6n^2+4n}(1+q^{4n+2})
\sum_{|k|\le n}(-1)^k a^{2k}q^{-2k^2},
\]
which is \eqref{twovar-even}.

For the odd part, put $j=\pm(2k+1)$ in \eqref{master-sum}.
Since $(-1)^j=-1$ for odd $j$,
\begin{equation}\label{twovar-odd-first}
P_a(q)\big|\sieve{2}{1}
=
-\sum_{k\ge0}\big(a^{2k+1}+a^{-2k-1}\big)
q^{(2k+1)^2}\gamma_{2k+1}.
\end{equation}
From \eqref{gamma-eval},
\[
q^{(2k+1)^2}\gamma_{2k+1}
=
q^{4k^2+8k+3}
\sum_{m\ge0}(-1)^mq^{6m^2+(12k+10)m}
(1+q^{4k+4m+4}).
\]
Substitution into \eqref{twovar-odd-first} gives
\[
P_a(q)\big|\sieve{2}{1}
=
\sum_{k,m\ge0}(-1)^{m+1}
\big(a^{2k+1}+a^{-2k-1}\big)
q^{4k^2+8k+6m^2+(12k+10)m+3}
(1+q^{4k+4m+4}).
\]
Now put $n=k+m$. Then
\[
4k^2+8k+6m^2+(12k+10)m
=
6n^2+10n-2k(k+1),
\]
and $(-1)^{m+1}=(-1)^{n+1}(-1)^k$. Therefore
\[
P_a(q)\big|\sieve{2}{1}
=
\sum_{n\ge0}(-1)^{n+1}q^{6n^2+10n+3}(1+q^{4n+4})
\sum_{k=0}^{n}(-1)^k
\big(a^{2k+1}+a^{-2k-1}\big)q^{-2k(k+1)},
\]
which is \eqref{twovar-odd}.
\end{proof}

\begin{proof}[Proof of Corollaries~\ref{cor:even} and~\ref{cor:odd}]
Set $a=1$ in \eqref{twovar-even} and \eqref{twovar-odd}. Then
$a^{2j}=1$ and $a^{2j+1}+a^{-2j-1}=2$, which give
\eqref{even-main} and \eqref{odd-main}, respectively.
\end{proof}

\section*{Acknowledgment}

We thank Jeremy Lovejoy for pointing out the earlier sources of the
Bailey pair and conjugate Bailey pair used in this paper.

\section*{Declarations}


\noindent\textbf{Conflict of interest.} The authors declare that they have no conflict of interest.

\noindent\textbf{Data availability.} Data sharing is not applicable to this article as no datasets were generated or analyzed during the current study.


\begin{thebibliography}{99}

\bibitem{AndrewsTheory}
G. E. Andrews, \emph{The Theory of Partitions}, Addison-Wesley, Reading, MA, 1976; reprinted by Cambridge University Press, Cambridge, 1998.

\bibitem{AndrewsQSeries}
G. E. Andrews, \emph{$q$-Series: Their Development and Application in Analysis,
Number Theory, Combinatorics, Physics, and Computer Algebra}, CBMS Regional
Conference Series in Mathematics, vol. 66, American Mathematical Society,
Providence, RI, 1986.

\bibitem{AndrewsBaileyChains}
G. E. Andrews, \emph{Bailey chains and generalized Lambert series: I. Four
identities of Ramanujan}, Illinois J. Math. \textbf{36} (1992), no.~2,
251--274. MR1156626.

\bibitem{AndrewsBachraoui}
G. E. Andrews and M. El Bachraoui, \emph{Congruences for two-color partitions with odd smallest part}, arXiv:2410.14190.

\bibitem{BanerjeeBringmann}
K. Banerjee and K. Bringmann, \emph{Proof of a conjecture of Andrews and Bachraoui on a Hecke sum}, arXiv:2605.10300v1.

\bibitem{GasperRahman}
G. Gasper and M. Rahman, \emph{Basic Hypergeometric Series}, 2nd ed., Encyclopedia of Mathematics and its Applications, vol. 96, Cambridge University Press, Cambridge, 2004.

\bibitem{LovejoyPartialTheta}
J. Lovejoy, \emph{Ramanujan-type partial theta identities and conjugate
Bailey pairs}, Ramanujan J. \textbf{29} (2012), no.~1--3, 51--67.
MR2994089.

\bibitem{MacMahon}
P. A. MacMahon, \emph{Combinatory Analysis}, Vol.~II, Cambridge University Press, Cambridge, 1916; reprinted in \emph{Combinatory Analysis, Volumes I and II}, AMS Chelsea Publishing, vol.~137, American Mathematical Society, Providence, RI, 1984.

\bibitem{WarnaarBailey}
S. O. Warnaar, \emph{50 years of Bailey's lemma}, in \emph{Algebraic
Combinatorics and Applications} (G\"ossweinstein, 1999), Springer, Berlin,
2001, pp. 333--347.

\end{thebibliography}
\end{document}